\documentclass[11pt]{article}
\usepackage{amssymb}
\usepackage{amsmath}
\usepackage{amsfonts}

\newtheorem{theorem}{Theorem}

\newtheorem{corollary}[theorem]{Corollary}

\newtheorem{definition}[theorem]{Definition}

\newtheorem{lemma}[theorem]{Lemma}

\newtheorem{remark}[theorem]{Remark}

\newenvironment{proof}[1][Proof]{\noindent\textbf{#1.} }{\ \rule{0.5em}{0.5em}}
\input{tcilatex}

\begin{document}

\title{Fractional Gamma process and fractional Gamma-subordinated processes}
\author{Luisa Beghin\thanks{%
Address: Department of Statistical Sciences, Sapienza University of Rome,
P.le A. Moro 5, I-00185 Roma, Italy. e-mail: \texttt{luisa.beghin@uniroma1.it%
}}}
\date{}
\maketitle

\begin{abstract}
We define and study fractional versions of the well-known Gamma subordinator
$\Gamma :=\{\Gamma (t),$ $t\geq 0\},$ which are obtained by time-changing $%
\Gamma $ by means of an independent stable subordinator or its inverse.
Their densities are proved to satisfy differential equations expressed in
terms of fractional versions of the shift operator (with fractional
parameter greater or less than one, in the two cases). As a consequence, the
fractional generalization of some Gamma subordinated processes (i.e. the
Variance Gamma, the Geometric Stable and the Negative Binomial) are
introduced and the corresponding fractional differential equations are
obtained.


\end{abstract}

\section{Introduction and preliminaries}

The Gamma subordinator $\Gamma (t),t>0$ is a very well-known process,
applied to many different fields such as, for example, engineering
reliability, maintenance theory, risk theory, option pricing and so on. It
can be considered as a particular case of the tempered stable subordinators
and thus used in financial modelling (see e.g. \cite{CON}).

We define here two fractional versions of $\Gamma $, obtained by a random
time-change by means of an independent stable subordinator $\mathcal{A}%
_{1/\nu }$, or, alternatively, the inverse stable subordinator $\mathcal{L}%
_{\nu }$ (see section 2.2 for their exact definitions). Thus we define, for
any $t\geq 0,$%
\begin{equation}
\left\{
\begin{array}{l}
\Gamma _{\nu }(t):=\Gamma (\mathcal{L}_{\nu }(t)),\quad 0<\nu <1 \\
\overline{\Gamma }_{\nu }(t):=\Gamma (\mathcal{A}_{1/\nu }(t)),\quad \nu >1%
\end{array}%
\right.  \label{pr6}
\end{equation}%
(the case $\nu =1$ corresponds to the standard Gamma process $\Gamma (t)$).
Only for $\nu >1$, $\overline{\Gamma }_{\nu }$ represents itself a L\'{e}vy
process, since it is obtained by subordinating a L\'{e}vy process to a
stable subordinator.

The processes defined in (\ref{pr6}) can be considered as fractional
versions of $\Gamma $, since we prove that, in both cases, their
distributions satisfy differential equations expressed in terms of a new
operator, that we call "fractional shift operator". We recall the definition
of the (integer order) shift operator: let $D_{x}^{n}:=d^{n}/dx^{n}$, for
any $n\in \mathbb{N}$, then%
\begin{equation}
e^{cD_{x}}f(x):=\sum_{n=0}^{\infty }\frac{c^{n}D_{x}^{n}}{n!}f(x)=f(x+c),
\label{shi}
\end{equation}%
for any analytic function $f:\mathbb{R}\rightarrow \mathbb{R}$ and $c\in
\mathbb{R}.$

The fractional counterpart of (\ref{shi}) is obtained by replacing the $n$%
-th order derivative by the $n$-fold iterated fractional derivative. We will
adopt the Caputo definition of fractional derivative, for $0<\nu \leq 1$,
i.e.%
\begin{equation}
D_{x}^{\nu }u(x):=\frac{1}{\Gamma (1-\nu )}\int_{0}^{x}\frac{1}{(x-s)^{\nu }}%
\frac{d}{ds}u(s)ds,  \label{ca}
\end{equation}%
while, for $\nu \geq 1$, we use the right sided fractional Riemann-Liouville
derivative on $\mathbb{R}^{+},$ i.e.%
\begin{equation}
\mathcal{D}_{-,x}^{\nu }u(t):=\frac{1}{\Gamma (m-\nu )}\left( -\frac{d}{dx}%
\right) ^{m}\int_{x}^{+\infty }\frac{1}{(s-x)^{1+\nu -m}}u(s)ds\text{,\qquad
}m-1<\nu <m,  \label{rl}
\end{equation}%
for $x>0$ (see (2.2.4) of \cite{KIL}, p.80).

Correspondingly we define the two operators $\mathcal{O}_{c,x}^{\nu },$ for $%
\nu \leq 1,$ and $\overline{\mathcal{O}}_{c,x}^{\nu },$ for $\nu \geq 1,$ as
fractional variants of (\ref{shi}) (see Definitions 1 and 2 below).

A generalized exponential operator of fractional order has been presented in
\cite{DATT} (see also \cite{BABU}): when applied to power functions, i.e. $%
f(x)=x^{k},$ $k\in \mathbb{N}$, it is proved to produce the so-called
Hermite-Kamp\'{e} de Feri\'{e}t polynomials. More recently, another
fractional version of (\ref{shi}) has been proposed in \cite{MIS} (where the
exponential is replaced by the Mittag-Leffler function).

By means of the operators $\mathcal{O}_{c,x}^{\nu }$ and $\overline{\mathcal{%
O}}_{c,x}^{\nu }$, we obtain, in the next section, the fractional
differential equations satisfied by the one-dimensional distributions of $%
\Gamma _{\nu }$ and $\overline{\Gamma }_{\nu }$, in the two ranges of the
fractional parameter $\nu $. The expedience of using (\ref{ca}) and (\ref{rl}%
) in the definitions of $\mathcal{O}_{c,x}^{\nu }$ and $\overline{\mathcal{O}%
}_{c,x}^{\nu }$ for $\nu <1$ and $\nu >1$, respectively, has been suggested
by the results on stable subordinators and their inverse: indeed it is known
that the law of $\mathcal{A}_{1/\nu }$ (that we will denote as $h_{1/\nu
}(x,t)$) satisfies the following fractional equation of order $\nu >1$%
\begin{equation*}
\mathcal{D}_{-,\,t}^{\nu }h_{1/\nu }=\frac{\partial }{\partial x}h_{1/\nu
},\quad \;x,\text{ }t\geq 0,\;h_{1/\nu }(x,0)=\delta (x)
\end{equation*}%
(with other appropriate initial conditions, see \cite{DOV} and \cite{BEG1}
for details). On the other hand the density of the process $\mathcal{L}_{\nu
}(t):=\inf \{z:\mathcal{A}_{\nu }(z)>t\}$ (denoted hereafter as $l_{\nu
}(x,t)$) satisfies the fractional equation of order $\nu <1$
\begin{equation*}
D_{\,t}^{\nu }l_{\nu }=-\frac{\partial }{\partial x}l_{\nu },\quad \;x,\text{
}t\geq 0,\text{ }l_{\nu }(x,0)=\delta (x),
\end{equation*}%
(see \cite{HAH}).

Moreover, we prove that the one-dimensional distributions of the processes
in (\ref{pr6}) satisfy, alternatively, a differential equation expressed in
terms of the fractional version of the operator $\mathcal{P}_{c,x}$, defined
as%
\begin{equation}
\mathcal{P}_{c,x}f(x):=\sum_{j=1}^{\infty }\frac{(-1)^{j+1}}{jc^{j}}%
D_{x}^{j}f(x),\text{ \qquad }c,x\in \mathbb{R}\text{,}  \label{a}
\end{equation}%
for any infinitely differentiable function $f.$

As a consequence of all the previous results, we derive in Section 3 the
differential equations satisfied by the fractional versions of the following
Gamma subordinated processes: the Variance Gamma process, the Geometric
Stable subordinator and the Negative Binomial process.

The Variance Gamma (hereafter VG) process (alternatively defined as Laplace
motion) is obtained by subordinating a Brownian motion to an independent
Gamma subordinator:%
\begin{equation*}
X(t):=B(\Gamma (t)),\qquad t\geq 0,
\end{equation*}%
where $B$ is a standard Brownian motion. The VG process is a particular case
of a symmetric geometric $\nu $-stable process, for $\nu =2,$ and it is
widely used in the financial theory, in order to model the logarithm of
stock prices (see, e.g. \cite{MAD}, \cite{KOT}). It has been already
proposed a fractional version of the VG process in \cite{KOZ}, defined by
subordinating a fractional Brownian motion $B_{H}$ to an independent Gamma
subordinator, i.e. as $X_{H}(t):=B_{H}(\Gamma (t)),$ where $H\in (0,1)$ is
the Hurst exponent. This process is useful to model hydraulic conductivity
fields in geophysics, as well as financial time series. We propose here
different fractional versions of $X$, defined as $X_{\nu }(t):=B(\Gamma
_{\nu }(t)),$ $t\geq 0,$ for $\nu <1$, and $\overline{X}_{\nu }(t):=B(%
\overline{\Gamma }_{\nu }(t)),$ $t\geq 0$, for $\nu >1$. Again, in the last
case, we get a L\'{e}vy process and the corresponding L\'{e}vy symbol is
obtained. Moreover, by definition, it is clear that the marginal
distributions of the fractional VG processes are scales mixtures of normal
laws: indeed it is%
\begin{equation}
X_{\nu }(t)\overset{d}{=}\Gamma _{\nu }(t)Z,\text{ for }\nu <1\text{ and }%
\overline{X}_{\nu }(t)\overset{d}{=}\overline{\Gamma }_{\nu }(t)Z,\text{ for
}\nu >1\text{,}  \label{re}
\end{equation}%
where $Z$ is a standard Gaussian variable and $\overset{d}{=}$ denotes the
equality of one-dimensional distributions. By comparing (\ref{re}) with
formula (1.5) of \cite{KOZ}, we can note that here $\Gamma _{\nu }$ and $%
\overline{\Gamma }_{\nu }$ play the same role of the generalized Gamma (or
Amoroso) random process $G_{t}^{2H}$ process (whose law is reported there in
(2.1)). Thus they can represent the stochastic variance or volatility, in
financial terms.

Also the Geometric Stable (hereafter GS) subordinator is widely studied and
applied, especially in financial contexts (see \cite{KOZ3}); it is one of
the special subordinators for which the potential measure has a decreasing
density, thus a wide potential theory has been established for it (see \cite%
{BOG}, \cite{SIK}). The GS process is defined as a stable subordinator
time-changed by means of a Gamma process (see (\ref{gs}) below). The
differential equation satisfied by its density has been obtained in \cite%
{BEG} and we generalize it to the fractional case.

The Negative Binomial (hereafter NB) process is a discrete valued process,
which can be defined, alternatively, as a compound Poisson process with
logarithmic jumps or as a mixed Poisson process (i.e. a Poisson process
subordinated to an independent Gamma subordinator, see, e.g., \cite{KOZ2}).
Through the first definition, a fractional version of the NB process has
been introduced in \cite{BEG2} and the corresponding densities are proved to
solve fractional recursive differential equations, which generalize the
Kolmogorov ones. By exploiting the mixing representation, we obtain here
alternative differential equations, involving the fractional shift operators.

\section{Fractional Gamma processes}

We first recall the following preliminary result, proved in \cite{BEG}: the
one-dimensional distribution of the Gamma subordinator $\Gamma (t),t\geq 0,$
of parameter $b>0,$ i.e.%
\begin{equation}
f_{\Gamma }(x,t):=\Pr \left\{ \Gamma (t)\in dx\right\} =\left\{
\begin{array}{l}
\frac{b^{t}}{\Gamma (t)}x^{t-1}e^{-bx},\qquad x\geq 0 \\
0,\qquad x<0%
\end{array}%
\right. ,  \label{gam}
\end{equation}%
satisfies the following Cauchy problem, for $x,t\geq 0,$%
\begin{equation}
\left\{
\begin{array}{l}
\frac{\partial }{\partial x}f_{\Gamma }=-b(1-e^{-\partial _{t}})f_{\Gamma }
\\
f_{\Gamma }(x,0)=\delta (x) \\
\lim_{|x|\rightarrow +\infty }f_{\Gamma }(x,t)=0%
\end{array}%
\right. ,  \label{res}
\end{equation}%
where $e^{-\partial _{t}}$ is the partial derivative version of the shift
operator defined in (\ref{shi}) and $\delta (x)$ is the Dirac delta function.

Then we need to introduce the definition of the \textit{fractional shift
operators}, for the two cases $\nu \leq 1$ and $\nu \geq 1.$

\begin{definition}
Let $f:\mathbb{R}^{+}\rightarrow \mathbb{R}$ be a continuous function with
fractional derivative $D_{x}^{\nu }$ defined in (\ref{ca}), for $\nu \in
(0,1],$ then%
\begin{equation}
\mathcal{O}_{c,x}^{\nu }f(x):=\sum_{n=0}^{\infty }\frac{c^{n}}{n!}%
\underbrace{D_{x}^{\nu }...D_{x}^{\nu }}_{n-times}f(x),  \label{fr1}
\end{equation}%
provided that the series converges.
\end{definition}

\

\begin{definition}
Let $f:\mathbb{R}^{+}\rightarrow \mathbb{R}$ be a continuous function with
fractional derivative $\mathcal{D}_{-,x}^{\nu }$ defined in (\ref{rl}), for $%
\nu \geq 1,$ then
\begin{equation}
\overline{\mathcal{O}}_{c,x}^{\nu }f(x):=\sum_{n=0}^{\infty }\frac{(-c)^{n}}{%
n!}\underbrace{\mathcal{D}_{-,x}^{\nu }...\mathcal{D}_{-,x}^{\nu }}%
_{n-times}f(x),  \label{fr2}
\end{equation}%
provided that the series converges.
\end{definition}

The semigroup property does not hold for the fractional derivatives $%
D_{x}^{\nu }$ and $\mathcal{D}_{-,x}^{\nu }$ and thus for the operators (\ref%
{fr1}) and (\ref{fr2}) cannot be used the formalism $e^{cD_{x}^{\nu }}$
adopted in \cite{DATT}.

It is easy to check that, for $\nu =1$, the fractional shift operator
defined in (\ref{fr1}) coincides with the standard shift operator in (\ref%
{shi}): indeed we get%
\begin{equation*}
\mathcal{O}_{c,x}^{1}f(x)=\sum_{n=0}^{\infty }\frac{c^{n}}{n!}%
D_{x}^{n}f(x)=e^{cD_{x}}f(x).
\end{equation*}%
On the other hand, for $\nu =1$, formula (\ref{fr2}) reduces to%
\begin{equation*}
\overline{\mathcal{O}}_{c,x}^{1}f(x):=\sum_{n=0}^{\infty }\frac{(-c)^{n}}{n!}%
(-1)^{n}D_{x}^{n}f(x)=e^{cD_{x}}f(x),
\end{equation*}%
since $\mathcal{D}_{-,x}^{n}=(-1)^{n}D_{x}^{n}$ (see (2.2.5) of \cite{KIL}).

We note that (\ref{fr1}) and (\ref{fr2}) do not coincide with the fractional
analogue of the Taylor's series expansion introduced in \cite{OSL} and its
generalizations presented in \cite{TRUJ}, \cite{JUM}.

\subsection{The case $\protect\nu <1$: the fractional Gamma process}

We consider now the first fractional Gamma process, defined as%
\begin{equation}
\Gamma _{\nu }(t):=\Gamma (\mathcal{L}_{\nu }(t))\text{,\qquad }\nu \in
(0,1],\text{ }t\geq 0,  \label{pr}
\end{equation}%
where $\Gamma $ is a Gamma process independent of $\mathcal{L}_{\nu }$ and
by $\mathcal{L}_{\nu }$ we denote the inverse of a stable subordinator $%
\mathcal{A}_{\nu }$ of index $\nu $ (with parameters $\mu =0,$ $\beta =1,$ $%
\sigma =(t\cos \pi \nu /2)^{1/\nu }$, in the notation of \cite{SAMO}). Thus,
by definition, $\mathcal{L}_{\nu }(t):=\inf \left\{ z\geq 0:\mathcal{A}_{\nu
}(z)>t\right\} \ $and we recall that
\begin{equation}
\mathbb{E}e^{-k\mathcal{L}_{\nu }(t)}=E_{\nu ,1}(-kt^{\nu }),\qquad k>0,
\label{elle2}
\end{equation}%
where
\begin{equation*}
E_{\nu ,\beta }(x)=\sum_{j=0}^{\infty }\frac{x^{j}}{\Gamma (\nu j+\beta )}%
,\quad \mathcal{R}(\nu )>0,\text{ }\beta ,x\in \mathbb{C}
\end{equation*}%
is the Mittag-Leffler function. We start by deriving the fractional equation
satisfied by the one-dimensional distribution of the process $\Gamma _{\nu }$
defined in (\ref{pr}).

\begin{theorem}
Let $f_{\Gamma _{\nu }}(x,t):=\Pr \{\Gamma _{\nu }(t)\in dx\}$, for $x,t\geq
0$, and $l_{\nu }(x,t):=\Pr \{\mathcal{L}_{\nu }(t)\in dx\},$ then the
density%
\begin{equation}
f_{\Gamma _{\nu }}(x,t)=\int_{0}^{\infty }f_{\Gamma }(x,z)l_{\nu }(z,t)dz
\label{pr3}
\end{equation}%
satisfies, for $\nu \in (0,1)$ and $t>1$, the following equation%
\begin{equation}
\frac{\partial }{\partial x}f_{\Gamma _{\nu }}=-b(1-\mathcal{O}_{-1,t}^{\nu
})f_{\Gamma _{\nu }},\qquad x\geq 0,  \label{pr4}
\end{equation}%
with initial condition%
\begin{equation}
f_{\Gamma _{\nu }}(0,t)=0.  \label{pr5}
\end{equation}

\begin{proof}
The initial condition is immediately satisfied by (\ref{pr3}), since $%
f_{\Gamma }(0,t)=0,$ for $t>1.$ In order to verify equation (\ref{pr4}) we
evaluate the Laplace transform of (\ref{pr3}), with respect to $x$, by
denoting $\widetilde{f}_{\Gamma _{\nu }}(\theta ,t):=\int_{0}^{\infty
}e^{-\theta x}f_{\Gamma _{\nu }}(x,t)dx,$%
\begin{eqnarray}
\widetilde{f}_{\Gamma _{\nu }}(\theta ,t) &=&\int_{0}^{\infty }\widetilde{f}%
_{\Gamma }(\theta ,z)l_{\nu }(z,t)dz  \label{sub} \\
&=&\int_{0}^{\infty }\exp \left\{ -z\log \left( 1+\frac{\theta }{b}\right)
\right\} l_{\nu }(z,t)dz  \notag \\
&=&[\text{by (\ref{elle2}) with }k=\log [1+\theta /b)]>0]  \notag \\
&=&E_{\nu ,1}\left( -\log \left( 1+\frac{\theta }{b}\right) t^{\nu }\right) .
\notag
\end{eqnarray}%
By taking the Laplace transform of the r.h.s. of (\ref{pr4}) we get%
\begin{eqnarray}
&&-b(1-\mathcal{O}_{-1,t}^{\nu })\widetilde{f}_{\Gamma _{\nu }}(\theta ,t)
\label{rs} \\
&=&-bE_{\nu ,1}\left( -\log \left( 1+\frac{\theta }{b}\right) t^{\nu
}\right) +b\sum_{n=0}^{\infty }\frac{(-1)^{n}}{n!}\underbrace{D_{t}^{\nu
}...D_{t}^{\nu }}_{n-times}E_{\nu ,1}\left( -\log \left( 1+\frac{\theta }{b}%
\right) t^{\nu }\right)  \notag \\
&=&[\text{by n applications of formula (2.4.58) of \cite{KIL}]}  \notag \\
&=&-bE_{\nu ,1}\left( -\log \left( 1+\frac{\theta }{b}\right) t^{\nu
}\right) +bE_{\nu ,1}\left( -\log \left( 1+\frac{\theta }{b}\right) t^{\nu
}\right) \sum_{n=0}^{\infty }\frac{\left( \log \left( 1+\frac{\theta }{b}%
\right) \right) ^{n}}{n!}  \notag \\
&=&-b\left[ 1-1-\frac{\theta }{b}\right] E_{\nu ,1}\left( -\log \left( 1+%
\frac{\theta }{b}\right) t^{\nu }\right) =\theta E_{\nu ,1}\left( -\log
\left( 1+\frac{\theta }{b}\right) t^{\nu }\right) ,  \notag
\end{eqnarray}%
which coincides with the Laplace transform of the l.h.s. of (\ref{pr4}),
i.e.
\begin{eqnarray}
\theta \widetilde{f}_{\Gamma _{\nu }}(\theta ,t) &=&\theta E_{\nu ,1}\left(
-\log \left( 1+\frac{\theta }{b}\right) t^{\nu }\right) -f_{\Gamma _{\nu
}}(0,t)  \label{ls} \\
&=&[\text{by (\ref{pr5}) }]  \notag \\
&=&\theta E_{\nu ,1}\left( -\log \left( 1+\frac{\theta }{b}\right) t^{\nu
}\right) .  \notag
\end{eqnarray}
\end{proof}
\end{theorem}

\begin{remark}
We can alternatively prove that the density (\ref{pr3}) satisfies the
following time-fractional equation, for any $\nu \in (0,1)$ and $t>1$:%
\begin{equation}
D_{t}^{\nu }f_{\Gamma _{\nu }}=-\mathcal{P}_{b,x}f_{\Gamma _{\nu }},\qquad
x\geq 0,  \label{re5}
\end{equation}%
with initial conditions%
\begin{equation}
\left. \frac{\partial ^{j}}{\partial x^{j}}f_{\Gamma _{\nu
}}(x,t)\right\vert _{x=0}=0,\qquad j=0,1,...  \label{re6}
\end{equation}%
where $D_{t}^{\nu }$ denotes, as usual, the Caputo fractional derivative.
Indeed by taking the Laplace transform of (\ref{re5}) we get
\begin{eqnarray}
\mathcal{L}\left\{ D_{t}^{\nu }f_{\Gamma _{\nu }}(\cdot ,t);\theta \right\}
&=&D_{t}^{\nu }E_{\nu ,1}\left( -\log \left( 1+\frac{\theta }{b}\right)
t^{\nu }\right)  \label{eq4} \\
&=&-\log \left( 1+\frac{\theta }{b}\right) E_{\nu ,1}\left( -\log \left( 1+%
\frac{\theta }{b}\right) t^{\nu }\right)  \notag \\
&=&-\mathcal{L}\left\{ \mathcal{P}_{b,x}f_{\Gamma _{\nu }}(\cdot ,t);\theta
\right\} ,  \notag
\end{eqnarray}%
where $\mathcal{L}\left\{ f(\cdot );\theta \right\} :=\int_{0}^{\infty
}e^{-\theta x}f_{\Gamma }(x,t)dx.$ The last equality in (\ref{eq4}) is
obtained by considering the well-known formula%
\begin{equation}
\mathcal{L}\left\{ D_{x}^{l}f(\cdot );\theta \right\} =\theta ^{l}\widetilde{%
f}(\theta )-\left. \sum_{j=0}^{l-1}\theta ^{j}D_{x}^{l-1-j}f(x)\right\vert
_{x=0}.  \label{aa}
\end{equation}%
Indeed, by the definition of $\mathcal{P}_{c,x}$ given in (\ref{a}), for $%
c=b $, we get%
\begin{eqnarray*}
\mathcal{L}\left\{ \mathcal{P}_{b,x}f_{\Gamma _{\nu }}(\cdot ,t);\theta
\right\} &=&\sum_{l=1}^{\infty }\frac{(-1)^{l+1}}{lb^{l}}\int_{0}^{\infty
}e^{-\theta x}D_{x}^{l}f_{\Gamma _{\nu }}(x,t)dx \\
&=&[\text{by (\ref{re6})}] \\
&=&\sum_{l=1}^{\infty }\frac{(-1)^{l+1}\theta ^{l}}{lb^{l}}E_{\nu ,1}\left(
-\log \left( 1+\frac{\theta }{b}\right) t^{\nu }\right) \\
&=&\log \left( 1+\frac{\theta }{b}\right) E_{\nu ,1}\left( -\log \left( 1+%
\frac{\theta }{b}\right) t^{\nu }\right) .
\end{eqnarray*}
\end{remark}

We note that the process $\Gamma _{\nu }$ is no longer a subordinator, since
it is clear from (\ref{sub}) that its density is not infinitely divisible.
To avoid this problem we present in the next section a fractional version of
the Gamma process which is still infinitely divisible, and, being
increasing, is also a subordinator.

We analyze here some properties of the process $\Gamma _{\nu }$, such as its
moments. The expected value is finite and can be obtained by taking its
Laplace transform and considering the well-known result%
\begin{equation*}
\int_{0}^{\infty }e^{-st}l_{\nu }(x,t)dt=s^{\nu -1}e^{-s^{\nu }x}.
\end{equation*}%
Indeed we get%
\begin{eqnarray*}
\mathcal{L}\left\{ \mathbb{E}\Gamma _{\nu }(t);s\right\} &=&\int_{0}^{\infty
}e^{-st}\int_{0}^{\infty }\mathbb{E}\Gamma (z)l_{\nu }(z,t)dzdt \\
&=&\frac{1}{b}\int_{0}^{\infty }z\int_{0}^{\infty }e^{-st}l_{\nu }(z,t)dtdz
\\
&=&\frac{s^{\nu -1}}{b}\int_{0}^{\infty }ze^{-xs^{\nu }}dz=\frac{1}{bs^{\nu
+1}}
\end{eqnarray*}%
and thus
\begin{equation*}
\mathbb{E}\Gamma _{\nu }(t)=\frac{t^{\nu }}{b\Gamma (\nu +1)},
\end{equation*}%
which, for $\nu =1$, reduces to the well-known expected value of the Gamma
process.

\begin{lemma}
The $r$-th absolute moments of $\Gamma _{\nu }$ are given by%
\begin{equation}
\mathbb{E}\Gamma _{\nu }(t)^{r}=\frac{1}{b^{r}}\sum_{k=0}^{r}\QATOPD[ ] {r}{k%
}\frac{k!t^{\nu k}}{\Gamma (\nu k+1)},\qquad r\in \mathbb{Z},  \label{mom}
\end{equation}%
where $\QATOPD[ ] {r}{k}$ denotes the (unsigned) Stirling numbers of the
first kind.

\begin{proof}
Recall that, for the Gamma subordinator,%
\begin{equation}
\mathbb{E}\Gamma (t)^{r}=\frac{1}{b^{r}}\frac{\Gamma (r+t)}{\Gamma (t)},
\label{lem}
\end{equation}%
so that we get%
\begin{eqnarray}
\mathcal{L}\left\{ \mathbb{E}\Gamma _{\nu }(\cdot )^{r};s\right\}
&=&\int_{0}^{\infty }e^{-st}\int_{0}^{\infty }\mathbb{E}\Gamma (z)^{r}l_{\nu
}(z,t)dzdt  \label{mom2} \\
&=&\frac{s^{\nu -1}}{b^{r}}\int_{0}^{\infty }\frac{\Gamma (r+z)}{\Gamma (z)}%
e^{-s^{\nu }z}dt  \notag \\
&=&\frac{s^{\nu -1}}{b^{r}}\int_{0}^{\infty }z^{(r)}e^{-s^{\nu }z}dt,  \notag
\end{eqnarray}%
where $z^{(r)}$ denotes the rising factorial defined as $z^{(r)}:=\Gamma
(r+z)/\Gamma (r).$ We recall the following expansion for the rising
factorials:%
\begin{equation}
z^{(r)}=\sum_{k=0}^{r}\QATOPD[ ] {r}{k}z^{k},  \label{ris}
\end{equation}%
therefore (\ref{mom2}) can be rewritten as%
\begin{eqnarray*}
\mathcal{L}\left\{ \mathbb{E}\Gamma _{\nu }(\cdot )^{r};s\right\} &=&\frac{%
s^{\nu -1}}{b^{r}}\sum_{k=0}^{r}\QATOPD[ ] {r}{k}\int_{0}^{\infty
}z^{k}e^{-s^{\nu }z}dt \\
&=&\frac{s^{\nu -1}}{b^{r}}\sum_{k=0}^{r}\QATOPD[ ] {r}{k}\frac{k!}{s^{\nu
k+\nu }}=\frac{1}{b^{r}}\sum_{k=0}^{r}\QATOPD[ ] {r}{k}\frac{k!}{s^{\nu k+1}}%
.
\end{eqnarray*}%
By inverting the Laplace transform we get (\ref{mom}).
\end{proof}
\end{lemma}

In order to obtain the variance of $\Gamma _{\nu }$ we choose $r=2$ in (\ref%
{mom}), so that we get
\begin{equation*}
var\left( \Gamma _{\nu }(t)\right) =\frac{2t^{2\nu }}{b^{2}\Gamma (2\nu +1)}+%
\frac{t^{\nu }}{b^{2}\Gamma (\nu +1)}-\frac{t^{2\nu }}{b^{2}\Gamma ^{2}(\nu
+1)}.
\end{equation*}%
Again, for $\nu =1$, we obtain the variance of the Gamma process.

\subsection{The case $\protect\nu >1$: the fractional Gamma subordinator}

The second fractional Gamma process we present here is defined as%
\begin{equation}
\overline{\Gamma }_{\nu }(t):=\Gamma (\mathcal{A}_{1/\nu }(t))\text{,\qquad }%
\nu >1,\text{ }t\geq 0,  \label{sec}
\end{equation}%
where $\Gamma $ is a Gamma process and $\mathcal{A}_{1/\nu }$ is the
(independent) stable subordinator of index $1/\nu $ (with parameters $\mu
=0, $ $\theta =1,$ $\sigma =(t\cos \pi /2\nu )^{\nu })$. It is well-known
that
\begin{equation}
\mathbb{E}e^{-k\mathcal{A}_{1/\nu }(t)}=e^{-k^{1/\nu }t},\qquad k>0.
\label{sec2}
\end{equation}

\begin{theorem}
Let $f_{\overline{\Gamma }_{\nu }}(x,t):=\Pr \{\overline{\Gamma }_{\nu
}(t)\in dx\}$, $x,t\geq 0$, and $h_{1/\nu }(x,t):=\Pr \{\mathcal{A}_{1/\nu
}(t)\in dx\},$ then the density%
\begin{equation}
f_{\overline{\Gamma }_{\nu }}(x,t)=\int_{0}^{\infty }f_{\Gamma
}(x,z)h_{1/\nu }(z,t)dz  \label{sec3}
\end{equation}%
satisfies, for $\nu >1$ and $t>1$, the following equation%
\begin{equation}
\frac{\partial }{\partial x}f_{\overline{\Gamma }_{\nu }}=-b(1-\overline{%
\mathcal{O}}_{-1,t}^{\nu })f_{\overline{\Gamma }_{\nu }},\qquad x\geq 0,
\label{sec4}
\end{equation}%
with initial condition%
\begin{equation}
f_{\overline{\Gamma }_{\nu }}(0,t)=0.  \label{sec5}
\end{equation}
\end{theorem}

\begin{proof}
The condition (\ref{sec5}) is immediately satisfied by (\ref{sec3}). As far
as equation (\ref{sec4}) is concerned, as in the proofs of Theorem 3, we
take the Laplace transform of the r.h.s. of (\ref{sec4}) which reads%
\begin{eqnarray*}
&&-b(1-\overline{\mathcal{O}}_{-1,t}^{\nu })\int_{0}^{\infty }\widetilde{f}%
_{\Gamma }(\theta ,z)h_{1/\nu }(z,t)dz \\
&=&-b(1-\overline{\mathcal{O}}_{-1,t}^{\nu })\int_{0}^{\infty }\exp \left\{
-z\log \left( 1+\frac{\theta }{b}\right) \right\} h_{1/\nu }(z,t)dz \\
&=&\text{[by (\ref{sec2}) with }k=\log [1+\theta /b)]\text{]} \\
&=&-b(1-\overline{\mathcal{O}}_{-1,t}^{\nu })\exp \left\{ -t\left( \log
\left( 1+\frac{\theta }{b}\right) \right) ^{1/\nu }\right\}  \\
&=&-b\exp \left\{ -t\left( \log \left( 1+\frac{\theta }{b}\right) \right)
^{1/\nu }\right\} +b\sum_{n=0}^{\infty }\frac{1}{n!}\underbrace{\mathcal{D}%
_{-,t}^{\nu }...\mathcal{D}_{-,t}^{\nu }}_{n-times}\exp \left\{ -t\left(
\log \left( 1+\frac{\theta }{b}\right) \right) ^{1/\nu }\right\}  \\
&=&[\text{by n applications of (2.2.15) in \cite{KIL}}] \\
&=&-b\exp \left\{ -t\left( \log \left( 1+\frac{\theta }{b}\right) \right)
^{1/\nu }\right\} +b\exp \left\{ -t\left( \log \left( 1+\frac{\theta }{b}%
\right) \right) ^{1/\nu }\right\} \sum_{n=0}^{\infty }\frac{1}{n!}\left(
\log \left( 1+\frac{\theta }{b}\right) \right) ^{n} \\
&=&-b\exp \left\{ -t\left( \log \left( 1+\frac{\theta }{b}\right) \right)
^{1/\nu }\right\} +b\exp \left\{ -t\left( \log \left( 1+\frac{\theta }{b}%
\right) \right) ^{1/\nu }\right\} \frac{b+\theta }{b},
\end{eqnarray*}%
which coincides with the Laplace transform of the l.h.s. of (\ref{sec4}),
i.e.%
\begin{eqnarray*}
\mathcal{L}\{\frac{\partial }{\partial x}f_{\overline{\Gamma }_{\nu }}(\cdot
,t);\theta \} &=&\theta \widetilde{f}_{\overline{\Gamma }_{\nu }}(\theta
,t)-f_{\overline{\Gamma }_{\nu }}(0,t) \\
&=&[\text{by (\ref{sec5})]} \\
&=&\theta \exp \left\{ -t\left( \log \left( 1+\frac{\theta }{b}\right)
\right) ^{1/\nu }\right\} .
\end{eqnarray*}
\end{proof}

\begin{remark}
We can alternatively prove that the density (\ref{sec3}) satisfies the
following time-fractional equation, for any $\nu \in (1,+\infty )$ and $t>1$:%
\begin{equation}
\mathcal{D}_{-,t}^{\nu }f_{\overline{\Gamma }_{\nu }}=-\mathcal{P}_{b,x}f_{%
\overline{\Gamma }_{\nu }},\qquad x\geq 0,  \label{re8}
\end{equation}%
with initial conditions%
\begin{equation}
\left. \frac{\partial ^{j}}{\partial x^{j}}f_{\overline{\Gamma }_{\nu
}}(x,t)\right\vert _{x=0}=0,\qquad j=0,1,...  \label{re9}
\end{equation}%
where $\mathcal{D}_{-,t}^{\nu }$ denotes the Riemann-Liouville fractional
derivative defined in (\ref{rl}). Indeed by taking the Laplace transform of (%
\ref{re8}) we get
\begin{eqnarray*}
\mathcal{L}\left\{ \mathcal{D}_{-,t}^{\nu }f_{\overline{\Gamma }_{\nu
}}(\cdot ,t);\theta \right\} &=&\mathcal{D}_{-,t}^{\nu }\exp \left\{
-t\left( \log \left( 1+\frac{\theta }{b}\right) \right) ^{1/\nu }\right\} \\
&=&-\log \left( 1+\frac{\theta }{b}\right) \exp \left\{ -t\left( \log \left(
1+\frac{\theta }{b}\right) \right) ^{1/\nu }\right\} \\
&=&[\text{by (\ref{aa}) and (\ref{re9})}] \\
&=&-\mathcal{L}\left\{ \mathcal{P}_{b,x}f_{\overline{\Gamma }_{\nu }}(\cdot
,t);\theta \right\} .
\end{eqnarray*}
\end{remark}

We remark that the process $\overline{\Gamma }_{\nu }$ defined in (\ref{sec}%
) is obtained by subordinating a L\'{e}vy process to the independent stable
subordinator $\mathcal{A}_{1/\nu }$. Thus it is itself a L\'{e}vy process
(see, e.g. \cite{APPL}, Theorem 1.3.25) and, being also real valued and
increasing, it is a subordinator. Its Laplace exponent can be evaluated
directly:
\begin{equation}
\psi _{\overline{\Gamma }_{\nu }}(\theta ):=-\frac{1}{t}\log \mathbb{E}%
e^{-\theta \overline{\Gamma }_{\nu }(t)}=-\frac{1}{t}\log \widetilde{f}_{%
\overline{\Gamma }_{\nu }}(\theta ,t)=\left( \log \left( 1+\frac{\theta }{b}%
\right) \right) ^{1/\nu },  \label{lp}
\end{equation}%
which reduce, for $\nu =1,$ to the Laplace exponent of $\Gamma .$ It can be
checked that (\ref{lp}) is a Bernstein function, by verifying that $%
(-1)^{n}d^{(n)}\psi _{\overline{\Gamma }_{\nu }}(\theta )/d\theta ^{(n)}\geq
0.$ Moreover we have $lim_{\theta \rightarrow 0}\psi _{\overline{\Gamma }%
_{\nu }}(\theta )=0$. Thus, by Theorem 1.3.4, p.45 in \cite{APPL}, for $\psi
_{\overline{\Gamma }_{\nu }}(\theta )$ there exists the following
representation%
\begin{equation*}
\psi _{\overline{\Gamma }_{\nu }}(\theta )=a+b\theta +\int_{0}^{\infty
}(1-e^{-y\theta })\lambda (dy),
\end{equation*}%
for a measure $\lambda $ s.t. $\int_{0}^{\infty }(y\wedge 1)\lambda
(dy)<\infty $, for all $\theta >0,$ $a,b\geq 0.$

Loosely speaking formula (\ref{lp}) implies that $\overline{\Gamma }_{\nu }$
grows more slowly than the standard Gamma subordinator, as $x\rightarrow
+\infty .$ Indeed $\Gamma $ increases at a logarithmic rate (see, for
example, \cite{LIN}).

The L\'{e}vy symbol $\eta _{\overline{\Gamma }_{\nu }}(u):=\log \mathbb{E}%
e^{iu\overline{\Gamma }_{\nu }(t)}/t$ of $\overline{\Gamma }_{\nu }$ can be
obtained by applying Proposition 1.3.27 in \cite{APPL}: for any $u\in
\mathbb{R},$%
\begin{equation*}
\eta _{\overline{\Gamma }_{\nu }}(u)=-\psi _{\mathcal{A}_{1/\nu }}(-\eta
_{\Gamma }(u)),
\end{equation*}%
where $\psi _{\mathcal{A}_{1/\nu }}(\theta )=\theta ^{1/\nu }$ is the
Laplace exponent of the stable subordinator and $\eta _{\Gamma }(u)=-\log
(1-iu/b)$ the L\'{e}vy symbol of $\Gamma $. Thus we get%
\begin{equation*}
\eta _{\overline{\Gamma }_{\nu }}(u)=-\left[ \log (1-iu/b)\right] ^{1/\nu },
\end{equation*}%
which, for $\nu =1$ reduces to $\eta _{\Gamma }.$

Finally we note that the process $\overline{\Gamma }_{\nu }$ does not
possess any finite moment, since the same is true for the subordinator $%
\mathcal{A}_{1/\nu }.$

\section{Applications to Gamma-subordinated processes}

We will define and analyze the fractional versions of well-known processes,
such as the VG, the GS and the NB processes, which are all expressed through
a random time change by the Gamma subordinator. By applying the previous
results, we will be able to derive the fractional equations satisfied by
their distributions, expressed by means of the fractional shift operators $%
\mathcal{O}_{-1,t}^{\nu }$ and $\overline{\mathcal{O}}_{-1,t}^{\nu }$.

\subsection{Fractional Variance Gamma processes}

We consider now the process obtained by the composition of a Brownian motion
with one of the two fractional versions of the Gamma process studied so far.
Thus we define%
\begin{equation}
\left\{
\begin{array}{c}
X_{\nu }(t):=B(\Gamma _{\nu }(t)),\qquad t\geq 0,\text{ }\nu \in (0,1) \\
\overline{X}_{\nu }(t):=B(\overline{\Gamma }_{\nu }(t)),\qquad t\geq 0,\text{
}\nu \in (1,\infty )%
\end{array}%
\right.  \label{mee}
\end{equation}%
where $B$ is a standard Brownian motion and $\Gamma _{\nu }$ and $\overline{%
\Gamma }_{\nu }$ are independent of $B$. For $\nu =1$, the processes in (\ref%
{mee}) coincides with the Variance Gamma process. Let, for simplicity, $b=1$
from now onwards. Let us denote the one-dimensional distributions of $X_{\nu
}$ and $\overline{X}_{\nu }$ as
\begin{equation}
f_{X_{\nu }}(x,t)=\int_{0}^{\infty }f_{B}(x,z)f_{\Gamma _{\nu
}}(z,t)dz,\qquad x,t\geq 0,\text{ }\nu \in (0,1),  \label{vg}
\end{equation}%
and%
\begin{equation}
f_{\overline{X}_{\nu }}(x,t)=\int_{0}^{\infty }f_{B}(x,z)f_{\overline{\Gamma
}_{\nu }}(z,t)dz,\qquad x,t\geq 0,\text{ }\nu >1,  \label{vg2}
\end{equation}%
where $f_{B}$ is the transition density of the standard Brownian motion $B.$

\begin{theorem}
The density (\ref{vg}) of $X_{\nu }$ satisfies the following fractional
differential equation%
\begin{equation}
\frac{1}{2}\frac{\partial ^{2}}{\partial x^{2}}f_{X_{\nu }}=(1-\mathcal{O}%
_{-1,t}^{\nu })f_{X_{\nu }},\qquad x\geq 0,\text{ }t>1,\qquad \nu \in (0,1),
\label{vg1}
\end{equation}%
while the density (\ref{vg2}) of $\overline{X}_{\nu }$ satisfies the equation%
\begin{equation}
\frac{1}{2}\frac{\partial ^{2}}{\partial x^{2}}f_{\overline{X}_{\nu }}=(1-%
\overline{\mathcal{O}}_{-1,t}^{\nu })f_{\overline{X}_{\nu }},\qquad x\geq 0,%
\text{ }t>1,\qquad \nu >1.  \label{vg3}
\end{equation}%
The boundary conditions for both equations are%
\begin{equation}
\left\{
\begin{array}{l}
\lim_{|x|\rightarrow \infty }f(x,t)=0 \\
\lim_{|x|\rightarrow \infty }\frac{\partial }{\partial x}f(x,t)=0%
\end{array}%
\right. ,  \label{con}
\end{equation}%
where $f:=f_{X_{\nu }},$\ for $\nu \in (0,1)$ and $f:=f_{\overline{X}_{\nu
}},$\ for $\nu \in (1,\infty ).$

\begin{proof}
The conditions (\ref{con}) can be easily checked by considering that their
analogues are satisfied by $f_{B}$. For $\nu \in (0,1),$ in order to prove (%
\ref{vg1}), we take the Fourier transform with respect to $x$ of its r.h.s.:
let%
\begin{equation*}
\widehat{f}(u):=\mathcal{F}\left\{ f(\cdot );u\right\} =\int_{-\infty
}^{+\infty }e^{iux}f(x)dx,
\end{equation*}%
then%
\begin{eqnarray}
\mathcal{F}\left\{ (1-\mathcal{O}_{-1,t}^{\nu })f_{X_{\nu }}(\cdot
,t);u\right\} &=&(1-\mathcal{O}_{-1,t}^{\nu })\widehat{f}_{X_{\nu }}(u,t)
\label{cl1} \\
&=&(1-\mathcal{O}_{-1,t}^{\nu })E_{\nu ,1}\left( -t^{\nu }\log \left( 1+%
\frac{u^{2}}{2}\right) \right)  \notag \\
&=&-\frac{u^{2}}{2}E_{\nu ,1}\left( -t^{\nu }\log \left( 1+\frac{u^{2}}{2}%
\right) \right) ,  \notag
\end{eqnarray}%
where for the last step we have performed some calculations similar to (\ref%
{rs}). For the l.h.s. of (\ref{vg1}), by considering the conditions (\ref%
{con}), we get instead%
\begin{equation}
\mathcal{F}\left\{ \frac{1}{2}\frac{\partial ^{2}}{\partial x^{2}}f_{X_{\nu
}}(\cdot ,t);u\right\} =-\frac{u^{2}}{2}\widehat{f}_{X_{\nu }}(u,t),
\label{cl2}
\end{equation}%
which coincides with (\ref{cl1}). Analogously, for $\nu >1,$ we get%
\begin{eqnarray}
\mathcal{F}\left\{ (1-\overline{\mathcal{O}}_{-1,t}^{\nu })f_{\overline{X}%
_{\nu }}(\cdot ,t);u\right\} &=&(1-\overline{\mathcal{O}}_{-1,t}^{\nu })%
\widehat{f}_{\overline{X}_{\nu }}(u,t) \\
&=&(1-\overline{\mathcal{O}}_{-1,t}^{\nu })\exp \left\{ -t\left( \log \left(
1+\frac{u^{2}}{2}\right) \right) ^{1/\nu }\right\}  \notag \\
&=&-\frac{u^{2}}{2}\exp \left\{ -t\left( \log \left( 1+\frac{u^{2}}{2}%
\right) \right) ^{1/\nu }\right\} ,  \notag
\end{eqnarray}%
which is equal to the analogue of (\ref{cl2}), with $f_{X_{\nu }}$ replaced
by $f_{\overline{X}_{\nu }}.$
\end{proof}
\end{theorem}

Alternatively we can prove the previous result directly, without resorting
to the Fourier transform, by considering the heat equation together with (%
\ref{pr4}), for $\nu \in (0,1)$, and (\ref{sec4}), for $\nu >1$: indeed, in
the first case, we get%
\begin{eqnarray*}
\frac{1}{2}\frac{\partial ^{2}}{\partial x^{2}}f_{X_{\nu }}(x,t) &=&\frac{1}{%
2}\frac{\partial ^{2}}{\partial x^{2}}\int_{0}^{\infty }f_{B}(x,z)f_{\Gamma
_{\nu }}(z,t)dz \\
&=&\int_{0}^{\infty }\frac{\partial }{\partial z}f_{B}(x,z)f_{\Gamma _{\nu
}}(z,t)dz \\
&=&\left[ f_{B}(x,z)f_{\Gamma _{\nu }}(z,t)\right] _{z=0}^{z=\infty
}-\int_{0}^{\infty }f_{B}(x,z)\frac{\partial }{\partial z}f_{\Gamma _{\nu
}}(z,t)dz \\
&=&[\text{by (\ref{pr5})}] \\
&=&-\int_{0}^{\infty }f_{B}(x,z)\frac{\partial }{\partial z}f_{\Gamma _{\nu
}}(z,t)dz,
\end{eqnarray*}%
which coincides with the right-hand side of (\ref{vg1}). We can obtain (\ref%
{vg3}) analogously, for $\nu >1.$

We analyze now the properties of the two versions of fractional VG process,
starting with their absolute moments.

\begin{corollary}
The absolute $q$-moments of $X_{\nu }$ are given by%
\begin{equation}
\mathbb{E}|X_{\nu }(t)|^{q}=\left\{
\begin{array}{l}
\frac{2^{r}}{\sqrt{\pi }}\Gamma \left( r+\frac{1}{2}\right) \sum_{k=0}^{r}%
\QATOPD[ ] {r}{k}\frac{k!t^{\nu k}}{\Gamma (\nu k+1)},\qquad \text{for }q=2r
\\
\frac{2^{r+\frac{1}{2}}}{\sqrt{\pi }}\Gamma \left( r+1\right)
\int_{0}^{\infty }\frac{\Gamma \left( r+z+\frac{1}{2}\right) }{\Gamma (z)}%
l_{\nu }(z,t)dz\qquad \text{for }q=2r+1%
\end{array}%
\right. ,  \label{mom3}
\end{equation}%
for $\nu <1$, while $\mathbb{E}|\overline{X}_{\nu }(t)|^{q}=\infty $, for
any $q\geq 1$, for $\nu >1.$
\end{corollary}

\begin{proof}
For $\nu <1,$ by considering (\ref{vg}) and taking into account the
well-known form of the absolute moments of the Brownian motion $B$, we can
write%
\begin{eqnarray}
\mathbb{E}|X_{\nu }(t)|^{q} &=&\int_{0}^{+\infty }|x|^{q}\int_{0}^{\infty
}f_{B}(x,z)f_{\Gamma _{\nu }}(z,t)dzdx  \label{las} \\
&=&\int_{0}^{\infty }\mathbb{E}|B(z)|^{q}f_{\Gamma _{\nu }}(z,t)dz  \notag \\
&=&\frac{\sqrt{2}^{q}}{\sqrt{\pi }}\Gamma \left( \frac{q+1}{2}\right)
\int_{0}^{\infty }z^{q/2}f_{\Gamma _{\nu }}(z,t)dz.  \notag
\end{eqnarray}%
For $q=2r$, by considering (\ref{mom}), with $b=1$, we immediately get the
first line in (\ref{mom3}). For $q=2r+1$ we rewrite (\ref{las}) as%
\begin{eqnarray}
&&\mathbb{E}|X_{\nu }(t)|^{q}  \notag \\
&=&\frac{\sqrt{2}^{q}}{\sqrt{\pi }}\Gamma \left( \frac{q+1}{2}\right)
\int_{0}^{\infty }\mathbb{E}|\Gamma (z)|^{q/2}l_{\nu }(z,t)dz,  \notag
\end{eqnarray}%
which, by (\ref{lem}), coincides with the second line in (\ref{mom3}).
\end{proof}

\begin{remark}
It is easy to check that, for $\nu =1$, formula (\ref{mom3}) reduces to
\begin{equation*}
\mathbb{E}|X(t)|^{q}=\frac{2^{q/2}}{\sqrt{\pi }}\Gamma \left( \frac{q}{2}+%
\frac{1}{2}\right) \frac{\Gamma \left( \frac{q}{2}+t\right) }{\Gamma (t)}%
\text{ \qquad }q=1,2....,
\end{equation*}%
which is the well-known formula for the absolute moments of the VG process
(see \cite{KOZ}, formula (2.2) for $H=1/2$). Indeed, for $q=2r,$ we get from
(\ref{mom3})%
\begin{eqnarray*}
\mathbb{E}|X_{\nu }(t)|^{2r} &=&\frac{2^{r}}{\sqrt{\pi }}\Gamma \left( r+%
\frac{1}{2}\right) \sum_{k=0}^{r}\QATOPD[ ] {r}{k}t^{k} \\
&=&[\text{by (\ref{ris})}] \\
&=&\frac{2^{r}}{\sqrt{\pi }}\Gamma \left( r+\frac{1}{2}\right) \frac{\Gamma
\left( r+t\right) }{\Gamma (t)},
\end{eqnarray*}%
while, for $q=2r+1$, it is%
\begin{equation*}
\mathbb{E}|X_{\nu }(t)|^{2r+1}=\frac{2^{r+\frac{1}{2}}}{\sqrt{\pi }}\Gamma
\left( r+1\right) \frac{\Gamma \left( r+t+\frac{1}{2}\right) }{\Gamma (t)},
\end{equation*}%
since $l_{\nu }(z,t)=\delta (z-t)$, for $\nu =1.$ It is evident from (\ref%
{mom3}) that the even-order absolute moments are never linear in $t$: in
particular the variance is given by%
\begin{equation*}
var(X_{\nu }(t))=\frac{t^{\nu }}{\Gamma (\nu +1)}.
\end{equation*}
\end{remark}

The most important feature of the fractional VG process is that, for $\nu
>1, $ it is a L\'{e}vy process and thus infinitely divisible for any $t$,
since it is obtained by subordinating the Brownian motion (which is a L\'{e}%
vy process) to the subordinator $\overline{\Gamma }_{\nu }$.

The L\'{e}vy symbol $\eta _{\overline{X}_{\nu }}(u)$ of $\overline{X}_{\nu }$
can be obtained by applying again Proposition 1.3.27 in \cite{APPL}: for any
$u\in \mathbb{R},$%
\begin{equation*}
\eta _{\overline{X}_{\nu }}(u)=-\psi _{\overline{\Gamma }_{\nu }}(-\eta
_{B}(u)),
\end{equation*}%
where $\psi _{\overline{\Gamma }_{\nu }}(\theta )=\left( \log \left(
1+\theta \right) \right) ^{1/\nu }$ is the Laplace exponent of $\overline{%
\Gamma }_{\nu }$ (see (\ref{lp}), for $b=1$) and $\eta _{B}(u)=-u^{2}/2$ the
L\'{e}vy symbol of $B.$ Thus we get%
\begin{equation}
\eta _{\overline{X}_{\nu }}(u)=-\left( \log \left( 1+\frac{u^{2}}{2}\right)
\right) ^{1/\nu }  \label{dic}
\end{equation}%
which, for $\nu =1,$ reduces to $\eta _{X}$ (see \cite{APPL}, p.57)$.$

From (\ref{dic}) it is evident that $\overline{X}_{\nu }$ cannot be
represented as difference of two independent Gamma suborinators, as happens
for $X.$ Let $\Gamma ^{\prime }$ and $\Gamma ^{\prime \prime }$ be two
independent Gamma random processes with parameters $a=1$, $b=\sqrt{2}$. Then
it is well-known that $X(t)\overset{d}{=}\Gamma ^{\prime }(t)-\Gamma
^{\prime \prime }(t)$, $t\geq 0.$ But, in the fractional case, for $\nu >1$,
the characteristic functions are respectively%
\begin{equation*}
\mathbb{E}e^{iu\overline{X}_{\nu }(t)}=\exp \left\{ -t\left( \log \left( 1+%
\frac{u^{2}}{2}\right) \right) ^{1/\nu }\right\} ,
\end{equation*}%
and%
\begin{equation*}
\mathbb{E}e^{iu\overline{\Gamma }_{\nu }^{\prime }(t)-iu\overline{\Gamma }%
_{\nu }^{\prime \prime }(t)}=\exp \left\{ -t\left[ \log (1-iu/\sqrt{2})%
\right] ^{1/\nu }-t\left[ \log (1+iu/\sqrt{2})\right] ^{1/\nu }\right\} ,
\end{equation*}%
for $\overline{\Gamma }_{\nu }^{\prime }$ and $\overline{\Gamma }_{\nu
}^{^{\prime \prime }}$ independent and defined in (\ref{sec}) with $b=\sqrt{2%
}.$ The previous expressions coincide only in the special case $\nu =1.$

\subsection{Fractional Geometric Stable subordinator}

The Geometric Stable (GS) subordinator of index $\alpha $ is defined by the
following subordinating relationship%
\begin{equation}
G_{\alpha }(t):=\mathcal{A}_{\alpha }(\Gamma (t)),\qquad t\geq 0,\text{ }%
0<\alpha \leq 1,  \label{gs}
\end{equation}%
where $\mathcal{A}_{\alpha }$ denotes a stable subordinator of index $\alpha
$ and $\Gamma $ an independent Gamma subordinator with $b=1$, for simplicity
(see \cite{SIK} and \cite{KOZ3}). Its Laplace exponent is
\begin{equation}
\psi _{G_{\alpha }}(\theta )=\log \left( 1+\theta ^{\alpha }\right)  \notag
\end{equation}%
so that its L\'{e}vy measure is equal to%
\begin{equation*}
\lambda (dx)=\alpha x^{-1}E_{\alpha ,1}(x)dx.
\end{equation*}%
For $\alpha =1,$ the process $G_{\alpha }$ reduces to the Gamma
subordinator. The fractional equation satisfied by the density of $G_{\alpha
}$ has been already obtained in \cite{BEG}, in the general case of the GS
process (i.e. not necessarily totally skewed to the right). It is expressed
in terms of the fractional Riesz-Feller derivative. We derive here an
analogous equation, in terms of the Caputo fractional derivative $%
D_{x}^{\alpha }$ of order $\alpha \in (0,1]$ and then we extend it to the
fractional version of the GS subordinator. We define the latter as

\begin{equation}
\left\{
\begin{array}{l}
G_{\alpha }^{\nu }(t):=\mathcal{A}_{\alpha }(\Gamma _{\nu }(t)),\qquad t\geq
0,\text{ }0<\nu <1, \\
\overline{G}_{\alpha }^{\nu }(t):=\mathcal{A}_{\alpha }(\overline{\Gamma }%
_{\nu }(t)),\qquad t\geq 0,\text{ }\nu >1,%
\end{array}%
\right.  \label{ai4}
\end{equation}%
where $\alpha \in (0,1]$, $\Gamma _{\nu }$ and $\overline{\Gamma }_{\nu }$
are independent of $B$. For $\nu =1$, the processes in (\ref{ai4}) reduce to
the GS subordinator. Let us denote the one-dimensional distributions of $%
G_{\alpha }^{\nu }$ and $\overline{G}_{\alpha }^{\nu }$ as
\begin{equation}
f_{G_{\alpha }^{\nu }}(x,t)=\int_{0}^{\infty }h_{\alpha }(x,z)f_{\Gamma
_{\nu }}(z,t)dz,\qquad x,t\geq 0,\text{ }\nu \in (0,1),  \label{mee2}
\end{equation}%
and%
\begin{equation}
f_{\overline{G}_{\alpha }^{\nu }}(x,t)=\int_{0}^{\infty }h_{\alpha }(x,z)f_{%
\overline{\Gamma }_{\nu }}(z,t)dz,\qquad x,t\geq 0,\text{ }\nu >1.
\label{mee3}
\end{equation}

\begin{theorem}
The density (\ref{mee2}) of $G_{\alpha }^{\nu }$ satisfies the following
doubly fractional differential equation%
\begin{equation}
D_{x}^{\alpha }f_{G_{\alpha }^{\nu }}=(1-\mathcal{O}_{-1,t}^{\nu
})f_{G_{\alpha }^{\nu }},\qquad x\geq 0,\text{ }t>1,\qquad \alpha \in (0,1],%
\text{ }\nu \in (0,1)  \label{mee4}
\end{equation}%
while the density (\ref{mee3}) of $\overline{G}_{\alpha }^{\nu }$ satisfies
the equation%
\begin{equation}
D_{x}^{\alpha }f_{\overline{G}_{\alpha }^{\nu }}=(1-\overline{\mathcal{O}}%
_{-1,t}^{\nu })f_{\overline{G}_{\alpha }^{\nu }},\qquad x\geq 0,\text{ }%
t>1,\qquad \alpha \in (0,1],\text{ }\nu >1,  \label{mee5}
\end{equation}%
The initial condition for both equations is $f(0,t)=0$ (where $%
f:=f_{G_{\alpha }^{\nu }},$ for $\nu \in (0,1)$ and $f:=f_{\overline{G}%
_{\alpha }^{\nu }},$ for $\nu >1$).
\end{theorem}

\begin{proof}
We can apply the result given in (\ref{res}) to obtain the differential
equation satisfied by the density of $G_{\alpha }$. By (\ref{gs}), the
density of the GS subordinator can be written as%
\begin{equation}
f_{G_{\alpha }}(x,t)=\int_{0}^{\infty }h_{\alpha }(x,z)f_{\Gamma }(z,t)dz,
\label{ai}
\end{equation}%
where $h_{\alpha }(x,t)$ is the density of $\mathcal{A}_{\alpha }(t),$ $%
t\geq 0.$ Then $f_{G_{\alpha }}$ satisfies the following space-fractional
differential equation%
\begin{equation}
D_{x}^{\alpha }f_{G_{\alpha }}=-(1-e^{-\partial _{t}})f_{G_{\alpha
/2}},\qquad x,t\geq 0,  \label{ai2}
\end{equation}%
with initial condition%
\begin{equation}
f_{G_{\alpha }}(0,t)=0.  \label{ai3}
\end{equation}%
In order to get (\ref{ai2}) we apply a well-known result on stable
subordinators: the density $h_{\alpha }$ satisfies, for $\alpha \in (0,1],$
the following equation
\begin{equation}
D_{x}^{\alpha }h_{\alpha }=-\frac{\partial }{\partial t}h_{\alpha },\quad
h_{\alpha }(x,0)=\delta (x),\qquad x,t\geq 0.
\end{equation}%
Thus%
\begin{eqnarray*}
D_{x}^{\alpha }f_{G_{\alpha }}(x,t) &=&\int_{0}^{\infty }D_{x}^{\alpha
}h_{\alpha }(x,z)f_{\Gamma }(z,t)dz \\
&=&-\int_{0}^{\infty }\frac{\partial }{\partial z}h_{\alpha }(x,z)f_{\Gamma
}(z,t)dz \\
&=&[h_{\alpha }(x,z)f_{\Gamma }(z,t)]_{z=0}^{\infty }+\int_{0}^{\infty
}h_{\alpha }(x,z)\frac{\partial }{\partial z}f_{\Gamma }(z,t)dz,
\end{eqnarray*}%
which, by considering (\ref{res}), with $b=1,$ gives (\ref{ai2}). The
initial condition is trivially satisfied since $h_{\alpha }(0,z)=0$.
Equations (\ref{mee4}) and (\ref{mee5}) can be easily obtained by
considering theorems 3 and 6.
\end{proof}

Also in this case, only for $\nu >1$, the fractional GS process is still a
subordinator and its L\'{e}vy symbol can be obtain as follows:%
\begin{eqnarray*}
\eta _{\overline{G}_{\alpha }^{\nu }}(u) &=&-\psi _{\overline{\Gamma }_{\nu
}}(-\eta _{\mathcal{A}_{\alpha }}(u)) \\
&=&-\left\{ \log [1+(-iu)^{\alpha }]\right\} ^{1/\nu },
\end{eqnarray*}%
since $\eta _{\mathcal{A}_{\alpha }}(u)=-(-iu)^{\alpha }.$

\subsection{Fractional Negative Binomial process}

The NB process $M(t),t>0,$ is a jump L\'{e}vy process with the following
distribution
\begin{equation}
q_{k}(t):=\Pr \left\{ M(t)=k\right\} =\binom{t+k-1}{k}p^{t}(1-p)^{k},\qquad
t>0,\text{ }k\in \mathbb{N},\text{ }p\in \left( 0,1\right) ,  \label{uno}
\end{equation}%
where $\binom{x}{k}$ is the generalized binomial coefficient defined, for
any $x\in \mathbb{R}$, as
\begin{equation*}
\binom{x}{k}:=\frac{(x)_{k}}{k!}=\frac{x(x-1)...(x-k+1)}{k!}
\end{equation*}%
and $(x)_{k}$ is the falling factorial.

The NB process has two alternative representations (see \cite{KOZ2}). The
first one is in terms of compound Poisson process: let $X_{j}$ be i.i.d.
random variables with discrete logarithmic distribution%
\begin{equation}
\Pr \left\{ X_{j}=k\right\} =-\frac{(1-p)^{k}}{k\ln p},\qquad k=1,2,...
\label{log}
\end{equation}%
for any $j=1,2,...$, then the process defined as%
\begin{equation}
M(t):=\sum_{j=1}^{N_{\lambda }(t)}X_{j},\qquad t\geq 0,  \label{due}
\end{equation}%
where $N_{\lambda }$ denotes an homogeneous Poisson process of parameter $%
\lambda =\ln (1/p)$ (independent from $X_{j},$ for any $j$), is proved to be
equivalent to the process defined by (\ref{uno}). Thus $M$ can be considered
as a particular case of the so called continuous-time random walks.

Another possible representation of the NB process is in terms of Cox process
with Gamma distributed directing measure, i.e. $N_{\Lambda (0,t]},$ $t\geq
0, $ where $\Lambda (0,t]$ is a Gamma r.v. with shape parameter $t$ and
scale parameter $(1-p)/p$. Thus we immediately get the following equality in
distribution
\begin{equation}
M(t)\overset{d}{=}N_{1}(\Gamma (t)),\qquad t\geq 0,  \label{coxx}
\end{equation}%
where $\Gamma $ has density (\ref{gam}) with $a=1,$ $b=p/(1-p).$

\begin{lemma}
\label{lm:eqNB} The distribution of the NB process of parameter $p,$ given
in (\ref{uno}), satisfies the following differential equation:%
\begin{equation}
\left\{
\begin{array}{l}
\frac{d}{dt}q_{0}(t)=\ln p\,q_{0}(t) \\
\frac{d}{dt}q_{k}(t)=\left[ \ln p+\sum_{j=0}^{k-1}\frac{1}{t+j}\right]
q_{k}(t),\qquad k\geq 1%
\end{array}%
\right. ,  \label{lem1}
\end{equation}%
for $t\geq 0,$ with initial conditions
\begin{equation*}
q_{k}(0)=\left\{
\begin{array}{c}
1,\qquad k=0 \\
0,\qquad k>0%
\end{array}%
\right. .
\end{equation*}%
\textbf{Proof }We take the derivative of (\ref{uno}), by considering that%
\begin{equation}
\binom{t+k-1}{k}=\frac{(t)^{(k)}}{k!},  \label{a1}
\end{equation}%
where $(x)^{(k)}$ is the rising factorial defined as $(x)^{(k)}:=\Gamma
(x+k)/\Gamma (x)$:%
\begin{eqnarray*}
\frac{d}{dt}q_{k}(t) &=&\frac{(1-p)^{k}}{k!}\frac{d}{dt}\left[ p^{t}(t)^{(k)}%
\right] \\
&=&\frac{(1-p)^{k}}{k!}\left[ p^{t}\ln p\,(t)^{(k)}+p^{t}\frac{d}{dt}%
(t)^{(k)}\right] \\
&=&\ln p\,q_{k}(t)+\frac{(1-p)^{k}}{k!}p^{t}(t)^{(k)}\left[ \Psi
^{(0)}(t+k)-\Psi ^{(0)}(t)\right] ,
\end{eqnarray*}%
where, in the last step, we have used the well-known result on the
derivative of the rising factorial and $\Psi ^{(0)}$ denotes the Digamma
function defined as $\Psi ^{(0)}(x):=d\ln \Gamma (x)/dx$ (see, for example,
\cite{CHOI}). By the properties of the Digamma function it is easy to see
that $\Psi ^{(0)}(t+k)-\Psi ^{(0)}(t)=\sum_{j=0}^{k-1}\frac{1}{t+j}$, so
that we get (\ref{lem1}). The initial conditions are trivially satisfied, by
considering that $(0)^{(k)}=0$ for any $k>0$ and $(0)^{(k)}=1$ for $k=0.$%
\hfill $\square $
\end{lemma}

We now define a fractional version of the NB process by substituting in (\ref%
{coxx}) the standard Gamma subordinator with its fractional counterpart
defined in (\ref{pr6}). Thus we set%
\begin{equation}
\left\{
\begin{array}{l}
M_{\nu }(t):=N_{1}(\Gamma _{\nu }(t)),\text{\qquad }\nu \in (0,1) \\
\overline{M}_{\nu }(t):=N_{1}(\overline{\Gamma }_{\nu }(t)),\qquad \nu \in
(1,\infty )%
\end{array}%
\right. ,  \label{af}
\end{equation}%
for any $t\geq 0.$ For $\nu =1$, $M_{\nu }(t):=M(t).$

From (\ref{pr}) and (\ref{sec}), it is clear that
\begin{equation}
\left\{
\begin{array}{l}
M_{\nu }(t)\overset{d}{=}M(\mathcal{L}_{\nu }(t))\text{,\qquad }\nu \in (0,1)
\\
\overline{M}_{\nu }(t)\overset{d}{=}M(\mathcal{A}_{1/\nu }(t))\text{,\qquad }%
\nu \in (1,\infty )%
\end{array}%
\right. .  \label{coxx2}
\end{equation}%
Unfortunately, due to the presence of non-constant coefficients in (\ref%
{lem1}), we are not able to derive its fractional analogue. By resorting to
the compound Poisson representation, it is proved in \cite{BEG2} that, for $%
\nu <1$, the distribution $q_{k}^{\nu }(t):=\Pr \{M_{\nu }(t)=k\}$ satisfies
the following birth-type (or Kolmogorov forward) fractional equations%
\begin{equation}
\left\{
\begin{array}{l}
D_{t}^{\nu }q_{0}^{\nu }=\ln pq_{0}^{\nu } \\
D_{t}^{\nu }q_{k}^{\nu }=\ln pq_{k}^{\nu }-\ln p\sum_{i=1}^{k}\frac{(1-p)^{i}%
}{i\ln p}q_{k-i}^{\nu },\qquad k>0%
\end{array}%
\right.  \label{vec}
\end{equation}%
with initial conditions%
\begin{equation*}
q_{0}^{\nu }(0)=1,\quad q_{k}^{\nu }(0)=0\text{, for all integer }k>0,
\end{equation*}%
Analogously, for $\nu >1$, the distribution $\overline{q}_{k}^{\nu }(t):=\Pr
\{\overline{M}_{\nu }(t)=k\}$ satisfies
\begin{equation}
\left\{
\begin{array}{l}
\mathcal{D}_{-,t}^{\nu }q_{0}^{\nu }=-\ln pq_{0}^{\nu } \\
\mathcal{D}_{-,t}^{\nu }q_{k}^{\nu }=-\ln pq_{k}^{\nu }+\ln p\sum_{i=1}^{k}%
\frac{(1-p)^{i}}{i\ln p}q_{k-i}^{\nu },\qquad k>0%
\end{array}%
\right.  \label{vec2}
\end{equation}%
with initial conditions%
\begin{equation*}
\overline{q}_{0}^{\nu }(0)=1,\quad \overline{q}_{k}^{\nu }(0)=0\text{, for
all integer }k>0.
\end{equation*}%
By applying the results of section 2, we can obtain alternative equations
satisfied by the distribution of the NB process and its fractional versions
defined in (\ref{af}), in terms of the shift and fractional shift operators.

\begin{theorem}
The distribution of the NB process $M$ solves the following fractional
differential equations, for $t>1,$%
\begin{equation}
e^{-D_{t}}q_{k}=\frac{1}{p}q_{k}-\frac{1-p}{p}q_{k-1},\qquad k\geq 0,
\label{gam3}
\end{equation}%
with initial condition%
\begin{equation*}
q_{0}(0)=1,\quad q_{k}(0)=0\text{, for any integer }k>0
\end{equation*}%
and $q_{-1}(t)=0$, for any $t.$
\end{theorem}

\begin{proof}
From (\ref{coxx}) we can write that%
\begin{equation}
q_{k}(t)=\int_{0}^{+\infty }p_{k}(z)f_{\Gamma }(z,t)dz,  \label{gam4}
\end{equation}%
where $p_{k}(t):=\Pr \{N_{1}(t)=k\}$ is the distribution of a Poisson
process with intensity $1$. By applying to (\ref{gam4}) the shift operator (%
\ref{shi}) with $c=-1$, we get, for any $k\geq 0,$%
\begin{eqnarray}
e^{-D_{t}}q_{k}(t) &=&\int_{0}^{+\infty }p_{k}(z)e^{-\partial _{t}}f_{\Gamma
}(z,t)dz  \label{dim} \\
&=&[\text{by (\ref{res}) with }b=p/(1-p)]  \notag \\
&=&\frac{1-p}{p}\int_{0}^{+\infty }p_{k}(z)\frac{\partial }{\partial z}%
f_{\Gamma }(z,t)dz+\int_{0}^{+\infty }p_{k}(z)f_{\Gamma }(z,t)dz  \notag \\
&=&[\text{integrating by parts}]  \notag \\
&=&\frac{1-p}{p}\left[ p_{k}(z)f_{\Gamma }(z,t)\right] _{z=0}^{z=+\infty }-%
\frac{1-p}{p}\int_{0}^{+\infty }\frac{\partial }{\partial z}%
p_{k}(z)f_{\Gamma }(z,t)dz+q_{k}(t)  \notag \\
&=&[\text{by (\ref{gam})-(\ref{res})}]  \notag \\
&=&-\frac{1-p}{p}\int_{0}^{+\infty }D_{z}p_{k}(z)f_{\Gamma }(z,t)dz+q_{k}(t),
\notag
\end{eqnarray}%
which coincides with (\ref{gam3}), by considering the well-known equation
satisfied by the Poisson distribution.
\end{proof}

By applying theorems 3 and 6, we obtain the fractional equations satisfied
by the distribution of the fractional NB processes defined in (\ref{coxx2}).

\begin{theorem}
The distribution $q_{k}^{\nu }(t)$ satisfies the following fractional
differential equations:%
\begin{equation}
\mathcal{O}_{-1,t}^{\nu }q_{k}^{\nu }=\frac{1}{p}q_{k}^{\nu }-\frac{1-p}{p}%
q_{k-1}^{\nu },\qquad k\geq 0,\text{ }t>1,\qquad \nu \in (0,1)  \label{gam5}
\end{equation}%
(where $\mathcal{O}_{-1,t}^{\nu }$ is defined in (\ref{fr1})). The
distribution $\overline{q}_{k}^{\nu }(t)$ satisfies%
\begin{equation}
\overline{\mathcal{O}}_{-1,t}^{\nu }\overline{q}_{k}^{\nu }=\frac{1}{p}%
\overline{q}_{k}^{\nu }-\frac{1-p}{p}\overline{q}_{k-1}^{\nu },\qquad k\geq
0,\text{ }t>1,\qquad \nu >1  \label{gam6}
\end{equation}%
(where $\overline{\mathcal{O}}_{-1,t}^{\nu }$ is defined in (\ref{fr2})) The
initial conditions for both equations are
\begin{equation*}
q_{0}(0)=1,\quad q_{k}(0)=0\text{, for any integer }k>0
\end{equation*}%
and $q_{-1}(t)=0$, for any $t,$ where we set $q:=q^{\nu }$, for $\nu <1$ and
$q:=\overline{q}^{\nu }$, for $\nu >1.$
\end{theorem}

\begin{proof}
By (\ref{af}) we get, for $\nu <1$,.%
\begin{equation*}
q_{k}^{\nu }(t)=\int_{0}^{+\infty }p_{k}(z)f_{\Gamma _{\nu }}(z,t)dz,
\end{equation*}%
thus we can write that%
\begin{eqnarray*}
\mathcal{O}_{-1,t}^{\nu }q_{k}^{\nu } &=&\int_{0}^{+\infty }p_{k}(z)\mathcal{%
O}_{-1,t}^{\nu }f_{\Gamma _{\nu }}(z,t)dz \\
&=&\left[ \text{by (\ref{pr4}), with }b=p/(1-p)\right] \\
&=&\frac{1-p}{p}\int_{0}^{+\infty }p_{k}(z)\frac{\partial }{\partial z}%
f_{\Gamma _{\nu }}(z,t)dz+\int_{0}^{+\infty }p_{k}(z)f_{\Gamma _{\nu
}}(z,t)dz.
\end{eqnarray*}%
We obtain equation (\ref{gam5}), by some steps similar to (\ref{dim}). The case $%
\nu >1$ can be treated analogously.

\
\end{proof}

As happens with the fractional VG process, by definition (\ref{coxx2}), it
is clear that, in the case $\nu >1$, also $\overline{M}_{\nu }$ is a L\'{e}%
vy process, since the same is true for $M$ and $\mathcal{A}_{1/\nu }$ is a
subordinator. The L\'{e}vy symbol can be obtain by evaluating the
characteristic function of the process:%
\begin{eqnarray*}
\mathbb{E}e^{iu\overline{M}_{\nu }(t)} &=&\sum_{k=0}^{\infty
}e^{iuk}q_{k}^{\nu }(t) \\
&=&\sum_{k=0}^{\infty }e^{iuk}\int_{0}^{+\infty }q_{k}(z)h_{1/\nu }(z,t)dz \\
&=&\int_{0}^{+\infty }\mathbb{E}e^{iuM(z)}h_{1/\nu }(z,t)dz \\
&=&\int_{0}^{+\infty }\left( \frac{p}{1-(1-p)e^{iu}}\right) ^{z}h_{1/\nu
}(z,t)dz \\
&=&\int_{0}^{+\infty }\exp \left\{ -z\log \frac{1-(1-p)e^{iu}}{p}\right\}
h_{1/\nu }(z,t)dz \\
&=&\exp \left\{ -t\left( \log \frac{1-(1-p)e^{iu}}{p}\right) ^{1/\nu
}\right\} .
\end{eqnarray*}%
Thus we get%
\begin{equation*}
\eta _{\overline{M}_{\nu }}(u)=-\left( \log \frac{1-(1-p)e^{iu}}{p}\right)
^{1/\nu },
\end{equation*}%
which coincides with%
\begin{equation*}
-\psi _{\mathcal{A}_{1/\nu }}(-\eta _{M}(u))=-\left( -\eta _{M}(u)\right)
^{1/\nu },
\end{equation*}%
where $\eta _{M}(u)=\log \frac{p}{1-(1-p)e^{iu}}$ (see formula (1.1) in \cite%
{KOZ2}).

For the reader's convenience we sum up the results on the L\'{e}vy symbols
obtained so far in the following table: recall that in all the cases below
it is $\nu >1.$ For $\nu =1$ we obtain the well-known L\'{e}vy symbols of
the corresponding non-fractional processes.

\

\begin{center}
$%
\begin{array}{ll}
\text{\textbf{Process}} & \text{\textbf{L\'{e}vy symbol}} \\
\text{Fractional Gamma }\overline{\Gamma }_{\nu } & \eta _{\overline{\Gamma }%
_{\nu }}(u)=-\left[ \log (1-iu/b)\right] ^{1/\nu } \\
\text{Fractional VG }\overline{X}_{\nu } & \eta _{\overline{X}_{\nu }}(u)=-%
\left[ \log \left( 1+\frac{u^{2}}{2}\right) \right] ^{1/\nu } \\
\text{Fractional GS }\overline{G}_{\alpha }^{\nu } & \eta _{\overline{G}%
_{\alpha }^{\nu }}(u)=-\left[ \log \left( 1+(-iu)^{\alpha }\right) \right]
^{1/\nu } \\
\text{Fractional NB }\overline{M}_{\nu } & \eta _{\overline{M}_{\nu }}(u)=-%
\left[ \log \frac{1-(1-p)e^{iu}}{p}\right] ^{1/\nu }%
\end{array}%
$
\end{center}

\

\end{document}